\documentclass{amsart}
\usepackage{amsmath,amssymb,graphicx,amsthm,bm,hyperref,mathrsfs,tikz}
\usepackage{graphicx}
\usepackage{hyperref}
\hypersetup{colorlinks=true,citecolor=blue,linkcolor=cyan}

\newtheorem{thm}{Theorem}[section]
\newtheorem{cor}[thm]{Corollary}
\newtheorem{lem}[thm]{Lemma}
\newtheorem{prop}[thm]{Proposition}
\theoremstyle{definition}
\newtheorem{defn}[thm]{Definition}
\theoremstyle{remark}
\newtheorem{rem}[thm]{Remark}
\numberwithin{equation}{section}

\newcommand{\Ff}{\mathbb{F}}
\newcommand{\Gal}{\textnormal{Gal}}
\newcommand{\Z}{\mathbb{Z}}
\newcommand{\F}{\mathcal{F}}
\newcommand{\I}{\mathcal{I}}

\newcommand{\Pp}{\mathbb{P}}
\newcommand{\Res}{\textnormal{Res}}
\newcommand{\ord}{\textnormal{ord}}
\newcommand{\Hh}{\mathcal{H}}
\newcommand{\G}{\mathcal{G}}
\newcommand{\Prob}{\textnormal{Prob}}
\newcommand{\supp}{\textnormal{supp}}

\title[Rational points on cyclic covers]{On the distribution of the rational points on cyclic covers in the absence of roots of unity
}

\author{Lior Bary-Soroker}
\address{School of Mathematical Science, Tel Aviv University, Ramat Aviv, Tel Aviv, 6997801, Israel }
\email{barylior@post.tau.ac.il}
\author{Patrick Meisner}
\address{School of Mathematical Science, Tel Aviv University, Ramat Aviv, Tel Aviv, 6997801, Israel }
\email{meisner@mail.tau.ac.il}

\begin{document}

\begin{abstract}
In this paper we study the number of rational points on curves in an ensemble of abelian covers of the projective line: Let $\ell$ be a prime, $q$  a prime power and consider the ensemble $\mathcal{H}_{g,\ell}$ of $\ell$-cyclic covers of $\Pp^1_{\Ff_q}$ of genus $g$.

We assume that $q\not\equiv 0,1\mod \ell$.
If $2g+2\ell-2\not\equiv0\mod (\ell-1){\rm ord}_\ell(q)$, then $\mathcal{H}_{g,\ell}$ is empty. Otherwise, the number of rational points on a random curve in $\mathcal{H}_{g,\ell}$ distributes as $\sum_{i=1}^{q+1} X_i$ as $g\to \infty$, where $X_1,\ldots, X_{q+1}$ are i.i.d.\ random variables taking the values $0$ and $\ell$ with probabilities $\frac{\ell-1}{\ell}$ and $\frac{1}{\ell}$, respectively.
The novelty of our result is that it works in the absence of a primitive $\ell$-th-root of unity, the presence of which was crucial in previous studies.
\end{abstract}

\maketitle

\section{Introduction}\label{intro}
For a given smooth projective curve $C$ defined over a finite field $\mathbb{F}_q$ of genus $g=g(C)$, the Hasse-Weil bound says that
\[
|\# C(\mathbb{F}_q) - (q+1)| \leq 2g \sqrt{q}.
\]
Here $C(\mathbb{F}_q)$ denotes the set of $\Ff_q$-rational points on $C$ and $q+1 = \# \mathbb{P}^1(\mathbb{F}_q)$. The problem we are after in this paper is the distribution of $\#C(\mathbb{F}_q)$ in families of covers of $\mathbb{P}^1_{\Ff_q}$ of genus $g$ in the limit $g\to \infty$.

Kurlberg and Rudnick \cite{KR09} study the ensemble of hyperelliptic curves $C$ of genus $g\to \infty$ under the assumption that $q$ is odd. They show that $\#C(\mathbb{F}_q) $ distributes as $\sum_{i=1}^{q+1} Y_i$, with $Y_1,\ldots, Y_{q+1}$ i.i.d.\ (independent and identically distributed) random variables taking the values $0,1,2$
with probabilities
$\frac{q}{2(q + 1)}$, $\frac{1}{q + 1}$, $\frac{q}{2(q+1)}$, respectively. (In \emph{loc.cit.}, a different normalization is used, the above formulation is similar to the one appearing in \cite[Theorem~1.1]{BDFL10a} in terms of the trace of Frobenius.)

The result of Kurlberg and Rudnick has been a subject to generalizations and extensions in many directions, see e.g.\ \cite{BJ, BDFL10b,CWZ}.
One direction of generalization is to consider abelian covers of $\mathbb{P}^1$: Artin-Schreier covers \cite{BDFLS12,BDFL16},
biquadratic covers \cite{LMM} in odd characteristic, cyclic \cite{BDFKLOW,BDFL10a,BDFL11, Meisner17,Xiong10} and abelian \cite{Meisner_arxiv} covers of exponent dividing $q-1$. The latter assumption is used in a crucial way in those papers: it implies that $\mathbb{F}_q$ contains a primitive root of unity of order equal to the exponent of the covers, which allows the application of Kummer theory.

To be more precise, for a prime power $q$, a non-negative  integer $g$, and a prime $\ell$
consider the moduli space \begin{equation}\label{moduli_space}
\Hh_{g,\ell} = \{\varphi \colon C \to \mathbb{P}^1_{\Ff_q} : \Gal(\Ff_q(C)/\Ff_q(\mathbb{P}^1_{\Ff_q})) \cong \Z/\ell\Z,\quad g(C)=g \},
\end{equation}
of $\ell$-cyclic covers of $\Pp^1_{\Ff_q}$. Here $\Ff_q(C)$ denotes the function field of $C$ and  $\Ff_q(\Pp^1_{\Ff_q}) = \Ff_q(X)$ is the field of rational functions.
The state-of-the-art result \cite[Theorem~1.3]{BDFKLOW} for primes $\ell$ not dividing $q$ says that if $\ell\mid q-1$ and if $C$ is  a random curve chosen uniformly from $\Hh_{g,\ell}$, then
\begin{equation}\label{eq:BDFKLOW}
\Prob\left(\#C(\Ff_q)=N\right) = \Prob\left(\sum_{i=1}^{q+1} Y_i = N\right) + O\left(\frac{1}{g}\right),
\end{equation}
as $g\to \infty$,
where  $Y_1,\ldots, Y_{q+1}$ are i.i.d.\ random variables taking the values $0,1$ and $\ell$ with probabilities $\frac{(\ell-1)q}{\ell(q+\ell-1)}$, $\frac{\ell-1}{q+\ell-1}$, and $\frac{q}{\ell(q+\ell-1)}$, respectively.

The primes $\ell$ to which \eqref{eq:BDFKLOW} may be applied are bounded by $q-1$; in particular, for $\Ff_2$ the result is empty, for $\Ff_3$ and $\Ff_5$ one must have $\ell = 2$, and the main term of \eqref{eq:BDFKLOW} coincides with that of \cite{KR09}. Of course, for larger fields (such as  $\Ff_7$) \eqref{eq:BDFKLOW} is more general.

Our main result treats $q\not\equiv 0,1\mod \ell$. Not every genus may be obtained in this case: If $n_q$ is the multiplicative order of $q$ modulo $\ell$, then the genus $g$ of an $\ell$-cyclic cover must satisfy the congruence
\begin{equation}\label{genus_congruence}
2g+2\ell -2 \equiv 0 \mod (\ell-1)n_q,
\end{equation}
see Proposition~\ref{genformprop1}.

\begin{thm}\label{thm:main}
Let $q$ be a prime power and  $\ell$ a prime such that $q\not\equiv 0,1\mod{\ell}$. Let $C$ be a random curve chosen uniformly from $\Hh_{g,\ell}$ (see \eqref{moduli_space}).
Then, as $g\to\infty$ satisfying the congruence relation \eqref{genus_congruence},
\begin{equation}\label{eq:Main}
\Prob\left( \#C(\Ff_q)=N\right) = \Prob\left(\sum_{i=1}^{q+1} X_i = N \right) +
O\left(q^{-\frac{(1-\epsilon)g}{\ell-1}}\right).
\end{equation}
Here the $X_1,\ldots, X_{q+1}$ are i.i.d.\ random variables taking the values $0$ and $\ell$ with probabilities $\frac{\ell-1}{\ell}$ and $\frac{1}{\ell}$, respectively.
\end{thm}

In \eqref{eq:BDFKLOW} each of the $q+1$ random variables $Y_i$ models the number of rational points on $C$ lying above the $i$-th rational point $x_i$ of $\Pp^1(\Ff_q)$ (under some arbitrary order). The number of points is $0$ if $x_i$ is inert in $C$  (i.e., the fiber is irreducible as an $\Ff_q$-scheme), $1$ if ramified, and $\ell$ if totally split. The probabilities are derived from the probabilities for the point to have each of these splitting types.

In \eqref{eq:Main}, when $q\not\equiv 0,1\mod\ell$, 
the points of $\Pp^1(\Ff_q)$ behave similarly to \eqref{eq:BDFKLOW} except for the fact that they cannot be ramified, see Remark \ref{genformrem}. So each of these $q$ points contributes either $0$ or $\ell$. The probabilities are derived from the heuristic that the Frobenius element is a random element of $\Z/\ell\Z$, and the point splits if and only if the Frobenius is trivial.

The error term in \eqref{eq:Main} decays exponentially in $g$ as the error term in \cite{KR09}, while in \eqref{eq:BDFKLOW} the error term decays linearly in $g$. In the latter results, the big error term comes from points of $\Pp^1(\Ff_q)$ that ramify in $C$, while in our setting the points  of $\Pp^1(\Ff_q)$ are always unramified in $C$.

We conclude the introduction with a few words on the methods and an outline of the proof. When $\ell \mid q-1$, the field $\Ff_q$ contains a primitive $\ell$-th root of unity $\zeta_\ell$. Hence the $\ell$-cyclic covers of $\Pp^1$ over $\Ff_q$ are birationally equivalent to a plane curve of the form
\begin{equation}\label{eq:SE}
Y^\ell = F(X), \qquad \mbox{$F$ is $\ell$-th-power free}.
\end{equation}
This leads to a parameterization of the moduli space $\Hh_{g,\ell}$. Then one derives an analytic formula for the number of rational points lying above a fixed point of $\Pp^1(\Ff_q)$ in terms of Dirichlet characters. The next step is to apply  generating function techniques to derive good asymptotic formulas for each  of the terms.

When $q\not\equiv0,1\mod\ell$, Equation~\eqref{eq:SE} generates a non-Galois and in particular non-abelian cover. We use the ray class group and explicit Galois descent to obtain a parametrization of the $\ell$-cyclic Galois covers (Section~\ref{classify}) together with a genus formula (Section~\ref{genformsec}). We control the number of rational points lying over the points of $\Pp^1(\Ff_q)$ (Proposition~\ref{numptsprop}). From this point, we follow a similar path as described above: We derive an analytic formula (Proposition~\ref{numptscor}) and then use generating function techniques (Section~\ref{SetCount}) to prove Theorem~\ref{thm:main} (Section~\ref{sec:proof}).

\subsection*{Acknowledgments}
The authors thank Dan Haran for a useful conversation on wreath product actions on fields, Chantal David for crucial comments, and Will Sawin for finding a critical mistake in a previous version.

L.B.S.\ was partially supported by a grant of the Israel Science Foundation and by the Simons CRM Scholars program. P.M.\ has received funding from the European Research Council under the European Union's Seventh Framework Programme (FP7/2007-2013) / ERC grant agreement n$^{\text{o}}$ 320755.

\section{Classifying the Curves}\label{classify}

In this section we use explicit class field theory to classify the $\ell$-cyclic covers $C\to \Pp^1_{\Ff_q}$ which in terms of function fields corresponds to $\Z/\ell \Z$-extensions of $K=\Ff_q(X)$.

\subsection{Class Field Theory Preliminaries}\label{CFTprel}

We begin by recalling the basic objects and results of class field theory which we use. For more extensive background on this topic we refer the reader to \cite{CF} or to \cite{HM}, for a more specific treatment in the case of function fields.

For now, we allow $K$ to be an arbitrary global field and $L$ a tamely ramified abelian extension of $K$. Denote $\mathcal{D}(K)$, $\mathcal{D}(L)$ as the groups of divisors of $K$ and $L$, respectively. Define the conorm map on the set of prime divisors in $\mathcal{D}(K)$ by
\[
i_{L/K}(P) = \sum_{\mathfrak{P}|P}e(\mathfrak{P}/P)\mathfrak{P},
\]
where the sum is over all primes of $L$ dividing $P$. We then extend $i_{L/K}$ linearly to all of $\mathcal{D}(K)$.

For any effective divisor $\mathfrak{m}$ of $K$, denote
\[
\mathcal{D}_{\mathfrak{m}}(L) = \{D\in \mathcal{D}(L) : \supp(D)\cap \supp(i_{L/K}(\mathfrak{m})) = \emptyset\}.
\]

\begin{rem}
If $L=K$ then $i_{L/K}$ is the identity map and $\mathcal{D}_{\mathfrak{m}}(K)$ is the set of divisors of $\mathcal{D}(K)$ that are coprime to $\mathfrak{m}$.
\end{rem}

Let $P$ be a prime of $K$ that is unramified in $L$. Then there is a unique automorphism $\sigma_P \in \Gal(L/K)$, called the \textbf{Frobenius automorphism}, such that
\[
\sigma_P(x) \equiv x^{N(P)} \mod{\mathfrak{P}}
\]
for all $x\in \mathcal{O}_{\mathfrak{P}}$ and all places $\mathfrak{P}$ of $L$ lying above $P$, where $N(P)$ is the norm of $P$ and $\mathcal{O}_{\mathfrak{P}}$ is the local ring at $\mathfrak{P}$. Let $\mathfrak{m}$ be an effective divisor of $K$ such that $L/K$ is unramified outside of $\mathfrak{m}$ and define the \textbf{Artin map} as
\begin{center}
\begin{tabular}{ c c c c }
$A_{L/K}$ : & $\mathcal{D}_{\mathfrak{m}}(K)$ & $\to$ & $\Gal(L/K)$ \\
& $D$ & $\mapsto$ & $\prod_{P}\sigma_P^{\ord_P(D)},$
\end{tabular}
\end{center}
where the product is over all primes of $K$ dividing $D$.

Define the \textbf{ray} of $L$ modulo $\mathfrak{m}$ as
\[
\mathcal{P}_{\mathfrak{m}}(L) = \{(f) : f\in L^{*}, f\equiv 1 \mod{\mathfrak{P}^{\ord_{\mathfrak{P}}(i_{L/K}(\mathfrak{m}))}} \mbox{ for all places $\mathfrak{P}$ of $L$} \}
\]
where we use the notation $(f)$ for the principle divisor of $f$. Notice that $\mathcal{P}_{\mathfrak{m}}(L)$ is a subgroup of $\mathcal{D}_{\mathfrak{m}}(L)$.

We say that $\mathfrak{m}$ is a \textbf{modulus} for $L$ if $\mathcal{P}_{\mathfrak{m}}(L)\subset \ker(A_{L/K})$. In this case, we can view $A_{L/K}$ as taking elements from the ray class group: $\mathcal{C}\ell_{\mathfrak{m}}(L) := \mathcal{D}_{\mathfrak{m}}(L)/\mathcal{P}_{\mathfrak{m}}(L)$. The following are the class field theory results which we need, borrowed from \cite{HM}.

\begin{prop}
\item
\begin{enumerate}\label{CFTprop}

\item $\mathfrak{m}$ is a modulus of $L$ if and only if $L/K$ is unramified outside of $\supp(\mathfrak{m})$.
\item If $L/K$ is unramified outside of $\supp(\mathfrak{m})$ then $A_{L/K} : \mathcal{C}\ell_{\mathfrak{m}}(K)\to \Gal(L/K)$ is surjective.
\item Let $L_1,L_2/K$ be two finite abelian extensions with modulus $\mathfrak{m}$. Then $L_1=L_2$ if and only if $\ker(A_{L_1/K}) = \ker(A_{L_2/K})$.
\item For every effective divisor $\mathfrak{m}$ and finite index subgroup $H$ of $\mathcal{C}\ell_{\mathfrak{m}}(K)$ there exists a finite abelian extension $L$ of modulus $\mathfrak{m}$ such that $\ker(A_{L/K})=H$. Moreover, $n\mathcal{C}\ell_{\mathfrak{m}}(K) \subset H$ where $n$ is the exponent of $\Gal(L/K)$.
\end{enumerate}
\end{prop}

\subsection{Abstract Parametrization of Geometric $\ell$-cyclic Extensions}\label{ltor}

From now on we restrict to the case $K=\Ff_q(X)$, $n=\ell$ a prime not dividing $q$, and we want to parametrize the set
\[
\mathcal{H}_{\mathfrak{m}} = \{ L/K \mbox{ $\ell$-cyclic Galois extension, unramified outside $\frak{m}$ and $L\cap \overline{\mathbb{F}}_q=\mathbb{F}_q$}\}.
\]
Let $n_q$ be the multiplicative order of $q$ modulo $\ell$. So, $n_q = [K(\mu_\ell):K]=[\Ff_q(\mu_{\ell}):\Ff_q]$ where $\mu_\ell$ is the group of $\ell$-th-roots of unity.

First we compute the ray class group in this setting.

\begin{lem}\label{ltorlem}
Let $\mathfrak{m}$ be an effective divisor of $K$, then the ray class group $\mathcal{C}\ell_{\mathfrak{m}}(K)$ is finite and
\[
\mathcal{C}\ell_{\mathfrak{m}}(K)/\ell \mathcal{C}\ell_{\mathfrak{m}}(K) \cong  \Z/\ell\Z \times \prod_{\substack{P\in \supp(\mathfrak{m}) \\ n_q|\deg(P)}} \Z/\ell\Z.
\]
\end{lem}

\begin{proof}
We may write $\mathcal{D}_{\mathfrak{m}} = \mathcal{D}^0_{\mathfrak{m}}(K) \times \Z$ with  $\mathcal{D}^0_{\mathfrak{m}}(K)$ the subgroup of divisors of degree $0$. Since $\mathcal{P}_{m}(K)\subseteq \mathcal{D}^0_{\mathfrak{m}}(K)$, we have
\[
\mathcal{C}\ell_{\mathfrak{m}}(K)/\ell \mathcal{C}\ell_{\mathfrak{m}}(K)  \cong \mathcal{C}\ell^0_{\mathfrak{m}}(K)/\ell \mathcal{C}\ell^0_{\mathfrak{m}}(K)  \times \Z/\ell\Z
\]
with $\mathcal{C}\ell^0_{\mathfrak{m}}(K)=\mathcal{D}^0_{\mathfrak{m}}/\mathcal{P}_{\mathfrak{m}}(K)$. So it suffices to prove that
\[
\mathcal{C}\ell^0_{\mathfrak{m}}(K)/\ell \mathcal{C}\ell^0_{\mathfrak{m}}(K) \cong   \prod_{\substack{P\in \supp(\mathfrak{m}) \\ n_q|\deg(P)}} \Z/\ell\Z.
\]

First, suppose $\mathfrak{m} = P$ for some prime $P$ of $K$. Since $\Ff_q[X]$ is a principle ideal domain,
we get that
\[
\mathcal{D}^0_{\mathfrak{m}}(K) = \{(f) : f \in K^*, \ord_P(f)=0\} = \{(f) : f \in K^*, f \not\equiv 0,\infty \mod{P}\}.
\]
Hence,
\begin{align*}
\begin{split}
\mathcal{C}\ell^0_{\mathfrak{m}}(K) = & \{(f) : f \in K^*, f \not\equiv 0 ,\infty \mod{P}\}/\{(f) : f \in K^*, f \equiv 1 \mod{P}\} \\
 \cong & \left(\Ff_q[X]/P\right)^* = \Ff^*_{q^{\deg(P)}}
\end{split}
\end{align*}

Since $\Ff^*_{q^{\deg(P)}}$ is cyclic, and its order is divisible by $\ell$ if and only if $n_q\mid \deg(P)$, we have
\begin{align}\label{ltoreq}
\mathcal{C}\ell^0_{\mathfrak{m}}(K)/\ell \mathcal{C}\ell^0_{\mathfrak{m}}(K) \cong \begin{cases}\Z/\ell\Z & n_q|\deg(P) \\ 0 & \mbox{otherwise.} \end{cases}
\end{align}

Likewise, if $\mathfrak{m} = eP$ for some positive integer $e$, then
\[
\mathcal{C}\ell_{\mathfrak{m}}(K) \cong \left(\Ff_q[X]/P^e\right)^* \cong \left(\Ff_q[X]/P\right)^*\times H
\]
for some $p$-group $H$ (where $p=\textnormal{char}(K)\neq \ell$). Hence \eqref{ltoreq} holds.

Finally, assume $\mathfrak{m} = \sum_{P}e_P P$ with  $e_P\geq 0$ and only finitely many nonzero. Then by the Chinese Remainder Theorem,
\[
\mathcal{D}^0_{\mathfrak{m}}(K) \cong \prod_{P\in \supp(\mathfrak{m})}\mathcal{D}^0_{e_PP}(K)
\quad \mbox{and}\quad
\mathcal{P}_{\mathfrak{m}}(K) \cong \prod_{P\in \supp(\mathfrak{m})}\mathcal{P}_{e_PP}(K)\]
and hence, by \eqref{ltoreq}
\[
\mathcal{C}\ell^0_{\mathfrak{m}}(K)/\ell \mathcal{C}\ell^0_{\mathfrak{m}}(K) = \prod_{\substack{P\in \supp(\mathfrak{m}) \\ n_q|\deg(P)}} \Z/\ell\Z,
\]
as needed.
\end{proof}

\begin{defn}\label{def:n-divisible}
We say a polynomial $F$ is \textbf{$n$-divisible} if
\begin{enumerate}
\item $n\mid \deg(P)$ for all $P\mid F$ and
\item $F$ is monic.
\end{enumerate}
\end{defn}

\begin{rem}
A trivial observation, which plays a crucial point in the establishment of the analytic formula below, is that if $f$ is $n$-divisible for $n>1$, then $f(x) \neq 0$ for all $x\in \Ff_q$.
\end{rem}

\begin{cor}\label{ltorcor}
For every effective divisor $\mathfrak{m}$ of $K$,
let $\Omega_1$ be the set of pairs $(F,b)$ with $F\in \Ff_q[X]$ an $n_q$-divisible polynomial with prime factors in ${\rm supp}(\mathfrak{m})$ which is not an $\ell$-th-power and $b\in \Ff_{q^{n_q}}^*$ both $F,b$ taken up to $\ell$-th-powers in the respective groups and let $\Omega_2$ be the set of  $\ell$-cyclic extensions $L/K$ of modulus $\mathfrak{m}$ such that $L\cap \bar{\Ff}_q=\Ff_q$. Then 
\[
|\Omega_1| = (\ell-1)|\Omega_2|.
\]
\end{cor}

\begin{proof}
By Proposition~\ref{CFTprop}, the $\ell$-cyclic extensions $L/K$ of modulus $\mathfrak{m}$ are in bijection with subgroups of index $\ell$ of $\mathcal{C}\ell_{\mathfrak{m}}(K)/\ell \mathcal{C}\ell_{\mathfrak{m}}(K)$. By Lemma~\ref{ltorlem}, the latter are in $1$-to-$(\ell-1)$ correspondence with non-zero elements of 
$$\Z/\ell\Z\times \prod_{\substack{P\in \supp(\mathfrak{m}) \\ n_q|\deg(P)}} \Z/\ell\Z.$$
We identify the first copy of $\Z/\ell\Z$ with $\Ff_{q^{n_q}}^* /(\Ff_{q^{n_q}}^* )^\ell$.

Now, a nonzero element $(b,(a_P))$ corresponds to the pair $(b,\prod_{P} P^{a_P})$ bijectively. To conclude the proof, we need to note that there is exactly one $\ell$-cyclic extension $L/K$ with $L\cap \bar{\Ff}_q \neq \Ff_q$ (namely, $L=\Ff_{q^\ell}$) and there are exactly $\ell-1$ pairs with $F$ an $\ell$-th-power, namely $(1,b)$, $b\in \Ff_{q^{n_q}}^*/(\Ff_{q^{n_q}}^*)^\ell$.
\end{proof}

\subsection{Explicit Correspondence}\label{basefield}
The correspondence described in Corollary~\ref{ltorcor} is given non-explicitly. In this section, we construct such a correspondence, using explicit Galois descent. We denote the to-be-constructed extension corresponding to $(F,b)$ by $L_{F,b}$.
The case $n_q=1$ is treated by Kummer theory:

\begin{prop}\label{kummerprop}
Assume $n_q=1$. Then $L_{F,b}=K(\sqrt[\ell]{bF})$.
\end{prop}

\begin{proof}
This is the main result of Kummer theory, cf.\   \cite[Proposition~\S14.37]{DF}.
\end{proof}

From now on  assume $n_q>1$. As mentioned above, the extension $K(\sqrt[\ell]{F})/K$ is not Galois and in particular not abelian.
To deal with this, we extend the base field in order to apply Kummer theory and then descend the extension (using wreath products) to describe the model over the original field.

Let $K' := \Ff_{q^{n_q}}(X)$. Then, by the definition of $n_q$, we have $\mu_{\ell} \subseteq  K'$. Here we can apply Kummer theory, but $F$ gives rise to many $\ell$-cyclic extensions which, in order to descend back to $K$, we need to describe explicitly together with the Galois action.

\begin{rem}\label{basefieldrem2}
One word on the terminology we use, since both $K$ and $K'$ are rational function fields in $X$, we shall use the term `prime polynomial' in $K$ (resp.\ $K'$) usually denoted by $P$ (resp.\ $\mathfrak{P}$) to indicate both a monic irreducible polynomial in $X$ with coefficients in $\Ff_q$ (resp.\  $\Ff_{q^{n_q}}$) and the corresponding prime of the field (which is not $X=\infty$).

We denote by $\phi$ the $q$-th-power Frobenius. We extend its action from $\Ff_{q^{n_q}}$ to $K'$ by letting it act trivially on $X$. So,
\begin{equation}\label{def:Frob}
\phi(a_nX^n+a_{n-1} X^{n-1} + \cdots + a_0 ) = a_n^qX^n+a_{n-1}^q X^{n-1} + \cdots + a_0^q
\end{equation}
and $\Gal(K'/K) = \left<\phi\right>$.
\end{rem}
Every $x\in \Ff_q$ is fixed by $\phi$, i.e., $x^q=x$. Thus,
\begin{equation}\label{eq:phiqx}
\phi(f)(x) = (f(x))^q, \qquad x\in \Ff_q, f\in \Ff_{q^{n_q}}[X].
\end{equation}

We begin with a lemma about factorization.

\begin{lem}\label{basefieldlem}

Let $P$ be a prime polynomial in $K$ and let $m=\gcd(n_q,\deg(P))$. Then there exist distinct prime polynomials $\mathfrak{P}_1,\dots,\mathfrak{P}_m$ in $K'$ such that
\[
P = \mathfrak{P}_1 \cdots \mathfrak{P}_m
\]
and
$\phi(\mathfrak{P}_i) = \mathfrak{P}_{i+1}$, where $\phi$ is as defined in \eqref{def:Frob} and with the convention that $\mathfrak{P}_{m+1}=\mathfrak{P}_1$.

\end{lem}

\begin{proof}

Since $K'/K$ is a cyclic unramified extension, each prime can be decomposed as
\[
P = \mathfrak{P}_1 \cdots \mathfrak{P}_g,
\]
with the $\mathfrak{P}_i$ pairwise distinct.
Since the Galois group $\Gal(K'/K) = \left<\phi\right>$ acts transitively on the primes above $P$, by relabeling the primes we may assume that $\mathfrak{P}_{i+1} = \phi(\mathfrak{P}_i)$. Finally,
\[
\Ff_{q^{n_q}}[X]/P = (\Ff_{q}[X]/P) \otimes \Ff_{q^{n_q}} = \Ff_{q^{\deg(P)}}\otimes \Ff_{q^{n_q}} = \prod_{i=1}^m \Ff_{q^{n_q\deg(P)/m}},
\]
whence $g=m$.
\end{proof}

We extend the above factorization of primes to $n_q$-divisible polynomials.

\begin{cor}\label{basefieldcor}

If $F$ is an $n_q$-divisible polynomial in $K$, then there exist coprime polynomials in $K'$, $F_1,\dots,F_{n_q}$ such that
\begin{align}\label{stablefact}
\begin{split}
F&=F_1\cdots F_{n_q} \\
F_{i+1} &= \phi(F_i)
\end{split}
\end{align}
where $\phi$ is defined in \eqref{def:Frob} and $F_{n_q+1}:=F_1$. We call \eqref{stablefact} a \textbf{$\mu_{\ell}$-stable factorization} and the $F_i$ the corresponding \textbf{$\mu_{\ell}$-stable factors}.

\end{cor}

\begin{proof}
Setting $F_i = \prod_{P|F} \mathfrak{P}^{\ord_P(F)}_i$ with $\mathfrak{P}_i$ as in Lemmas \ref{basefieldlem} suffices.
\end{proof}

From now on we fix our pair $(F,b)$ and explicitly construct a $\Z/\ell\Z$ extension $L = L_{F,b}$. Since $F$ is taken up to $\ell$-th powers, we may assume that  $F$ is an $\ell$-th-power free polynomial in $K$ that is $n_q$-divisible.
We also fix notation: Let $F_1,\ldots, F_{n_q}$ be the $\mu_{\ell}$-stable factorization factors of $F$ as in \eqref{stablefact}, let $b_i=b^{q^{i-1}}$, and
\[
 u_i = \sqrt[\ell]{b_iF_i} ,
\]
for $i=1,\ldots, n_q$. We put
\[
E' = K(u_1,\ldots, u_{n_q}).
\]

Recall that the wreath product $A\wr B$ of two finite groups $A$ and $B$ is defined as the semi-direct product $A^B\rtimes B$ with $B$ acting on $A^B$ by permuting the indices. To be more precise, if for example $B =\Z/n_q\Z$ and $x=1\in B$ is a generator, then
\[
(-x) (v_{1}, \ldots, v_{n_q}) x = (v_{2}, \ldots, v_{n_q}, v_1).
\]

\begin{lem}\label{basefieldlem2}

Let $\zeta_\ell\in \mu_{\ell}$ be a primitive $\ell$-th-root of unity. Then
\begin{enumerate}

\item  $\Gal(E'/K')=\left(\Z/\ell\Z\right)^{n_q}$ with action given by
\begin{equation} \label{actionofv}
(h_1, \ldots, h_{n_q}).u_i = \zeta_{\ell}^{q^{i-1}h_i} u_i, \qquad i=1,\ldots, n_q.
\end{equation}

\item
The action of $\phi$ on $K'$ as defined in \eqref{def:Frob} may be lifted to an action on $E'$ by setting
\begin{equation}\label{actionofphi}
\phi(u_i) = u_{i+1},
\end{equation}
with our usual convention $u_{n_q+1} := u_1$.

\item The extension $E'/K$ is Galois with $\Gal(E'/K) \cong \Z/\ell\Z \wr \Z/n_q\Z$.

\item The fixed field $E$ of  $\left<\phi\right>$ in $E'$ is
\[
E = K(u_1+\cdots + u_{n_q}).
\]
\end{enumerate}

\end{lem}

\begin{proof}

For $(1)$, it is clear that $K'(u_1,\ldots, u_{n_q})\subseteq E'$ since $K'(u_j) = L_{e_j} \subseteq E'$. The other inclusion  is obvious since $L_{\mathbf{v}}\subseteq K'(u_1,\ldots, u_{n_q})$ for all $\mathbf{v}$. This proves that $E'=K'(u_1,\ldots, u_{n_q})$. Now, since the $F_i$-s are pairwise co-prime, we have that $\Gal(E'/K')=\left(\Z/\ell\Z\right)^{n_q}$. We may choose coordinates for $(\Z/\ell \Z)^{n_q}$ to have the stated action.

$(2)$ and $(3)$ are well known consequences of $(1)$, but we give a full account because we need the precise actions.

For $(2)$, lift $\phi$ in any way  to a map $\psi$ from $E'$ to its algebraic closure. Since
\[
\psi(u_i)^{\ell} = \psi(u_i^\ell) = \psi(b_iF_i) = \phi(b_iF_i) = b_{i+1}F_{i+1},
\]
we have that $\psi(u_i) = \zeta_{\ell}^{v_{i+1}} u_{i+1}$, for some tuple $(v_1,\ldots, v_{n_q})\in (\Z/\ell \Z)^{n_q}$ (as usual $v_{n_q+1}=v_1$). From this we conclude that $\psi\colon E'\to E'$ and in particular $E'/K$ is a Galois extension. To get the desired lift, we  change $\psi$ by an element of $\Gal(E'/K')$:
\[
(-v_{1},-v_2q^{-1},\ldots, -v_{n_q}q^{1-n_q})\psi,
\]
and this lift maps each $u_i$ to $u_{i+1}$. By abuse of notation, we denote the latter lift also by $\phi$.

For $(3)$, by $(1)$ and $(2)$, we have that
\[
\begin{split}
(\phi^{-1} (v_1,\ldots, v_{n_q}) \phi) (u_i) &= \phi^{-1} (v_1,\ldots, v_{n_q}) . u_{i+1}=\phi^{-1}\zeta_\ell^{q^{i}v_{i+1}} = \zeta_{\ell}^{q^{i-1}v_{i+1}} u_i\\
&= (v_2,\ldots, v_{n_q}, v_1).u_i.
\end{split}
\]
This shows that the Galois group $\Gal(E'/K)$ is the wreath product.

For $(4)$, put $u= \sum_{i=1}^n u_i$. Since  $K(u)\subseteq E$, by the fundamental theorem of Galois theory to show equality, it suffices to show that $K(u)$ is not fixed by any element $\sigma$ in $\Z/\ell\Z \wr \Z /n_q \Z$  which is not in $\left<\phi\right>$, which has the form $\sigma = ((h_1,\ldots, h_{n_q}), \phi^k)$ with $(h_1,\ldots, h_{n_q})$ a nonzero vector in $(\Z /\ell\Z)^{n_q}$. Assume by contradiction that $\sigma(u)\in K(u)$.
Now, by \eqref{actionofv} and since $u$ is fixed by $\phi$,
\[
\sigma (u) = \sum_{i}\zeta_\ell^{q^{i-1}h_i} u_i.
\]
Since $\sigma(u)$ is also fixed by $\phi$, all the powers of $\zeta_\ell$ must be  equal; i.e., $\sigma(u) = \zeta_{\ell}^a u$ with $a \equiv q^{i-1}h_i\mod\ell$ for all $i$. But then $\zeta_\ell^a \in K(u)$ and thus fixed by $\phi$, which implies that $a\equiv 0 \mod \ell$. So all $h_i \equiv0\mod\ell$, contradiction.
\end{proof}

Now that we have computed the Galois group, to study $\ell$-cyclic subextensions we need to find $\ell$-cyclic quotients of the wreath product.
\begin{lem}\label{subclassifylem1}
The wreath product $G=\Z/\ell\Z\wr\Z/n_q\Z$ has a unique $\ell$-cyclic quotient. More precisely, for $H\lhd G$ we have $G/H\cong \Z/\ell \Z$ if and only if $H = \tilde{H}\rtimes\Z/n_q\Z$, where
\[
\tilde{H} = \{ h = (h_1,\dots,h_{n_q}) : h_1+\dots+h_{n_q} \equiv 0 \mod{\ell}\}\leq \left(\Z/\ell\Z\right)^{n_q}.
\]

\end{lem}

\begin{proof}
Let $G'$ be the commutator of $G$.
Direct computation shows that
\[
G' \leq \{(h,0) : h_1+\dots+h_{n_q} = 0\}
\]
and since the quotient by the right hand side is abelian we have an equality.
Here $h=(h_1,\dots,h_{n_q})\in\left(\Z/\ell\Z\right)^{n_q}$.  Hence we have the isomorphism
\[
\begin{split}
G/G' &\to \Z/\ell\Z\times\Z/n_q\Z \cong \Z/\ell n_q\Z \\
(h,\sigma)&\mapsto \Big(\sum_j h_j, \sigma\Big)
\end{split}
\]
which implies that $G/G'$ has a unique $\ell$-cyclic quotient as a cyclic group, and so does  $G$. Moreover this quotient is given by $(h,\sigma)\mapsto \sum_j h_j$, and so  the kernel  is $\tilde{H}\rtimes \Z/n_q\Z$.
\end{proof}

We enumerate all the $\ell$-cyclic subextensions of $E'/K'$ by nonzero vectors  $\mathbf{v}=(v_1,\dots,v_{n_q})$ of elements in $\mathbb{Z}/\ell \mathbb{Z}$. We write
\begin{equation}\label{eq:DefFv}
F_{\mathbf{v}} := \prod_{i=1}^{n_q} (b_iF_i)^{v_i'},
\end{equation}
where  $0\leq v_i'\leq \ell-1$ is a minimal non-negative representative of $v_i$, $i=1,\ldots, n_q$ and we let
\begin{equation}\label{eq:DefLv}
L_{\mathbf{v}} = K'(\sqrt[\ell]{F_{\mathbf{v}}})
\end{equation}
be the corresponding Kummer extension.

\begin{lem}\label{subclassifylem2}
Let $\tilde{H}\leq H \leq \Z/\ell\Z\wr\Z/n_q\Z \cong \Gal(E'/K)$ be as in Lemma \ref{subclassifylem1} and let $L=E'^H$ and $\tilde{L}=E'^{\tilde{H}}$. Then,
\begin{align}
\label{Lv0}\tilde{L} &= L_{\mathbf{v_0}}, \\
\label{ExtensionofScalars}\tilde{L} &= L \Ff_{q^{n_q}},\quad \mbox{and}\\
\label{explicitequation}L &= \Ff_q(X)[Y]/\Bigg(\prod_{j=0}^{\ell-1} \Big( Y - \sum_{k=0}^{n_q-1} \zeta_{\ell}^{jq^k} \sqrt[\ell]{F_{\mathbf{v}_k}}\Big)\Bigg),
\end{align}
where  $\mathbf{v}_k =  (q^{k},q^{k-1},\dots,q^{k+1-n_q})$.
\end{lem}

\begin{proof}
By \eqref{eq:DefFv}, \eqref{actionofv}, and since $q^{j-1}v_j\equiv 1\mod \ell$,  we have
\[
h\left(\sqrt[\ell]{F_{\mathbf{v}_0}}\right) = \prod_{j=1}^{n_q} \zeta_\ell^{q^{j-1}h_jv_j} \sqrt[\ell]{F_{\mathbf{v}_0}} = \zeta_\ell ^{\sum_{j=1}^{n_q}h_j } \sqrt[\ell]{F_{\mathbf{v}_0}} = \sqrt[\ell]{F_{\mathbf{v}_0}},
\]
for all $h=(h_1,\ldots, h_{n_q})\in \tilde{H}$
Therefore, $L_{\mathbf{v}_0} \subseteq \tilde{L}$. On the other hand, $[L_{\mathbf{v}_0}:K']=\ell = [\tilde{L}:K']$, hence $\tilde{L}=L_{\mathbf{v}_0}$.

Since $\tilde{H} = H\cap (\Z/\ell\Z)^{n_q}$ and $(\Z/\ell\Z)^{n_q} \cong \Gal(E'/K')$, the fundamental theorem of Galois theory implies that
\[
\tilde{L} = LK' = L\Ff_{q^{n_q}}.
\]

Finally, let $u\in L$ be the trace of $\sqrt[\ell]{F_{\mathbf{v}_0}}\in \tilde{L}$ in the extension $\tilde{L}/L$. Since $\Gal(\tilde{L}/L)$ is generated by the restriction of $\phi$,
by \eqref{actionofphi}, we have
\[
u = \sum_{k=0}^{n_q-1} \sqrt[\ell]{F_{\mathbf{v}_k}}.
\]
The conjugates of $u$ over $K$ are $(j,0,\ldots, 0) u$, $j\in \Z/\ell\Z$ which by \eqref{actionofv} satisfy
\[
(j,0,\ldots,0)u = \sum_{k=0}^{n_q-1} \zeta_{\ell}^{jq^{k}} \sqrt[\ell]{F_{\mathbf{v}_k}}.
\]
Thus $u$ generates the $\ell$-cyclic extension $L/K$ and has
\[
\prod_{j=0}^{\ell-1} \left( Y - \sum_{k=0}^{n_q-1} \zeta_{\ell}^{jq^{k}} \sqrt[\ell]{F_{\mathbf{v}_k}}\right)
\]
as its minimal polynomial.
\end{proof}
We summarize the construction in the following diagram.

\begin{center}
\begin{tikzpicture}[node distance = 1cm, auto]
      \node (K) {$K$};
      \node (L) [above of=K] {$L$};
      \node (E) [above of=L] {$E$};
      \node (K') [right of=K] {$K'$};
      \node (L') [above of=K'] {$L_{\mathbf{v}}$};
      \node (E') [above of=L'] {$E'$};
      \draw[-] (K) to node {} (L);
      \draw[-] (L) to node {} (L');
      \draw[-] (L) to node {} (E);
      \draw[-] (K) to node {} (K');
      \draw[-] (K') to node {} (L');
      \draw[-] (L') to node {} (E');
      \draw[-] (E) to node {} (E');
      \end{tikzpicture}
\end{center}

Now we are ready to introduce an explicit description of the correspondence of Corollary \ref{ltorcor}.

\begin{prop}\label{subclassifyprop}
The correspondence
\[
\begin{split}
\Omega_1 &\to \Omega_2\\
(F,b)&\mapsto L=L_{F,b}
\end{split}
\]
induces an $(\ell-1)$-to-$1$ map between the sets $\Omega_1,\Omega_2$ defined  in Corollary \ref{ltorcor}. 
\end{prop}

\begin{proof}
Let $\frak m$ be an effective divisor and let $F$ and $G$ be two $n_q$-divisible polynomials that are supported on $\frak m$.
For each prime polynomial $P\in \Ff_q[X]$ of degree divisible by $n_q$, let $P=\frak P_1\cdots \frak P_{n_q}$ be a  $\mu_{\ell}$-stable factorization.
Let $F = F_1\cdots F_{n_q}$ and $G=G_1\cdots G_{n_q}$ be the corresponding $\mu_\ell$-stable factorizations of $F$ and $G$. So, $\mathfrak{P}_1$ appears only in  $F_1$ if $P\mid F$ and only in $G_1$ if $P\mid G$.

Assume that $L_{F,b} = L_{G,c}$. By \eqref{Lv0}, \eqref{ExtensionofScalars}, and \eqref{eq:DefLv} there exists $0\leq r\leq \ell-1$ such  that
$F_{\mathbf{v}_0}$ and $G_{\mathbf{v}_0}^r$ are equal up to $\ell$-th-powers in $K'=\Ff_{q^{n_q}}(X)$. Comparing leading coefficients gives that $b^{n_q} = c^{n_q r}$ and so $b = c^r$ up to $\ell$-th-powers.

Let $P$ be a prime dividing $F$ (resp.\ $G$) with multiplicity $\alpha$ (resp.\ $\beta$). Then $\frak P_1$ divides $F_1$ (resp.\ $G_1$) with the same multiplicity. So $\alpha\equiv r\beta \mod \ell$. This implies that $F=G^r$ up to $\ell$-th-powers.
So we got that up to respective $\ell$-th powers, $(F, b) = (G^r , c^r )$. This implies the correspondence between the pairs and extensions is  $(\ell-1)$-to-$1$ and by Corollary~\ref{ltorcor} we also obtain all the extensions.
\end{proof}

\begin{rem}\label{rmk:modelforcovers}
Let us translate the above construction in terms of covers and equations.
Let $C_{F,b}$ be the smooth projective model of $L_{F,b}$. By \eqref{explicitequation},  $C_{F,b}$ is birationally equivalent to the affine plane curve
\[
\Bigg\{\prod_{j=0}^{\ell-1} \left( Y - \sum_{k=0}^{n_q-1} \zeta_{\ell}^{jq^k} \sqrt[\ell]{F_{\mathbf{v}_k}}\right) = 0\Bigg\}.
\]
By \eqref{Lv0} and \eqref{ExtensionofScalars}, $C_{F,b}$ is geometrically  birationally equivalent to
\begin{equation}\label{rmk:modelforcovers}
\{ Y^\ell = F_{\mathbf{v}_0}\}.
\end{equation}
\end{rem}

\begin{rem}

If $n_q=1$, then $F_{\mathbf{v}_0}=bF$ and the characteristic polynomial becomes
\[
\prod_{j=0}^{\ell-1} \left(Y- \zeta_{\ell}^i \sqrt[\ell]{bF}\right) = Y^{\ell}-bF
\]
which recovers the statement in Kummer theory.

\end{rem}

\section{Genus Formula}\label{genformsec}
From now on we fix an $n_q$-divisible polynomial $F$ which is $\ell$-th-power free and $b\in \Ff_{q^{n_q}}^*$. We let $L=L_{F,b}$ be as in \eqref{explicitequation} and $C=C_{F,b}$ the smooth projective model of $L$. We also assume that
\[
n_q\neq1.
\]
Our goal is to give a formula for the genus $g=g(C)$ of $C$.

The genus of projective curves is preserved under separable base change, so the genus  $g$ equals the genus of the smooth projective curve which is birationally equivalent to the affine plane curve
\eqref{rmk:modelforcovers}. By the Riemann-Hurwitz formula the genus of the curve in \eqref{rmk:modelforcovers} may be computed in terms of a factorization of the $\ell$-th-power free polynomial $F_{\mathbf{v}_0}$ as explained below.

Since $F$ is $\ell$-th-power free, there exist monic, square-free polynomials $f_1,\dots,f_{\ell-1}$ which are pairwise coprime such that
 \[
F = f_1f_2^2\cdots f_{\ell-1}^{\ell-1}.
\]
Since $F$ is $n_q$-divisible, all the $f_i$ are $n_q$-divisible as well. Hence, by Corollary~\ref{basefieldcor}, for each $i$, we have a $\mu_\ell$-stable factorization of  $f_i$:
\begin{equation}\label{eq:f_i}
f_i = f_{i,1}\cdots f_{i,n_q}
\end{equation}
We now  define
\begin{align}\label{stablefact2}
F_j := f_{1,j}f_{2,j}^2\cdots f_{\ell-1,j}^{\ell-1},
\end{align}
which gives us a $\mu_\ell$-stable factorization of $F$, namely
\[
F=F_1\cdots F_{n_q},
\]
and $F_{i} = \phi^{i-1}(F_1)$, with $\phi$ the $q$-th-power Frobenuis.
Taking $v_i$ to be the minimal non-negative   integer with $v_i \equiv q^{1-i} \mod \ell$, we get by \eqref{eq:DefFv} and \eqref{stablefact2} that
\begin{align}\label{F_v}
F_{\mathbf{v}_0} = b'  \prod_{j=1}^{n_q} F_j^{v_j} = b' \prod_{i=1}^{\ell} \prod_{j=1}^{n_q} f_{i,j}^{iv_j},
\end{align}
for some $b'\in \Ff_{q^{n_q}}^*$.
We apply the Riemann-Hurwitz formula which in this setting gives that
\begin{equation}\label{eq:genus1}
2g+2\ell-2 = (\ell-1) \sum_{i=1}^{\ell}\sum_{j=1}^{n_q} \deg(f_{i,j}) + \begin{cases} 0 & \mbox{if } \deg(F_{\mathbf{v}_0}) \equiv 0 \mod{\ell} \\ \ell-1 & \mbox{otherwise} \end{cases}.
\end{equation}
(cf.\ \cite[Equation (1.2)]{Meisner17} with our $\ell$ replacing $r$ in \emph{loc.cit.} and using the fact that $\ell$ is coprime to all of the $v_i$).

Since $\deg(f_{i,j})=\deg(f_{i,1})$ for all $j$,
\[
\deg(F_{\mathbf{v}_0}) \equiv \sum_{i=1}^{\ell-1} \sum_{j=1}^{n_q} iq^{1-j}\deg(f_{i,j}) \equiv \sum_{i=1}^{\ell-1}i\deg(f_{i,1})\sum_{j=1}^{n_q} q^{1-j} \mod{\ell}.
\]
Since we assume $n_q\not=1$ and since $q^{n_q}\equiv 1\mod\ell$, we have
\[
\sum_{j=1}^{n_q} q^{1-j} \equiv q\cdot \frac{1-q^{-n_q}}{1-q^{-1}} \equiv 0 \mod{\ell}.
\]
Furthermore, by \eqref{eq:f_i}, $\sum_{j=1}^{n_q} \deg(f_{i,j}) = \deg(f_i)$. Plugging the above computations into \eqref{eq:genus1} yields the genus formula for $C$:
\begin{align}\label{genform}
2g+2\ell-2 = (\ell-1) \sum_{i=1}^{\ell} \deg(f_i).
\end{align}

\begin{rem}\label{genformrem}
The Riemann-Hurwitz formula and \eqref{genform} imply that the only primes that ramify in $\Ff_q(C)$ are those that divide $F$. Consequently,  no linear primes ramify in $\Ff_q(C)$ if $n_q\not=1$. Hence, all the points of $\Pp^1(\Ff_q)$ are unramified in $C$.

\end{rem}

We use \eqref{genform} to parametrize the moduli space $\mathcal{H}_{g,\ell}$ as defined in \eqref{moduli_space}. Define the following set
\begin{equation}\label{eq:FqD}
\F_{n_q}(D) = \left\{(f_1,\dots,f_{\ell-1})\in \Ff_q[X]\ :\ \parbox{15em}{$f_1,\ldots, f_{\ell-1}$ are square free, pairwise coprime, $n_q$-divisible, and
$\deg(f_1\cdots f_{\ell-1})=D$}\right\}.
\end{equation}
Each tuple $(f_1,\dots,f_{\ell-1})\in\F_{n_q}(D)$ corresponds to the $\ell$-th-power free $n_q$-divisible polynomial $F= f_1f_2^2\cdots f_{\ell-1}^{\ell-1}$.
By \eqref{genform}, the genus of a corresponding curve $C=C_{F,b}$ satisfies:
\begin{equation}\label{eq:gtoD}
D = \frac{2g+2\ell-2}{\ell-1}.
\end{equation}
Thus, by Corollary \ref{ltorcor}, the correspondence $(F,b)\mapsto C_{F,b}$ defines an $(\ell-1)$-to-$1$ correspondence between the sets
\[
\F_{n_q}(D) \times \Ff_{q^{n_q}}^*/(\Ff_{q^{n_q}}^*)^{\ell} \to \Hh_{g,\ell}.
\]
Since $n_q$-divisible polynomials have degree divisible by $n_q$, we have that $\F_{n_q}(D) = \emptyset$ if and only if $D\not \equiv 0 \mod n_q$. Thus we immediately get

\begin{prop}\label{genformprop1}
For $n_q>1$, $\Hh_{g,\ell}=\emptyset$ if $2g+2\ell-2 \not\equiv 0 \mod{(\ell-1)n_q}$.
\end{prop}

When $D\equiv 0\mod n_q$ or equivalently $2g+2\ell-2\equiv 0 \mod(\ell-1)n_q$,  it suffices to count the number of points on curves when we parameterize by $\F_{n_q}(D)$ instead of by $\Hh_{g,\ell}$.
More rigorously,
\begin{prop}\label{H_gtoF_n}
Let $g, D$  be related by \eqref{eq:gtoD} and assume that $n_q>1$ and that $2g+2\ell-2\equiv 0 \mod(\ell-1)n_q$. Choose a random $(f_1,\ldots, f_{\ell-1})$ uniformly in $\F_{n_q}(D)$ and a random $b\in \Ff_{q^{n_q}}^*$ and put $F$ and $C_{F,b}$ as above. Choose a random $C$ uniformly in $\Hh_{g,\ell}$. Then, for every $N\geq 0$
\[
\Prob(\#C_{F,b}(\Ff_q) = N) = \Prob(\#C(\Ff_q) = N).
\]
\end{prop}

\section{Number of Points Formula}\label{numpts}
The goal of this section is to give an analytic formula for the number of rational points on cyclic covers.
Let $\pi \colon C\to \Pp^1_{\Ff_q}$ be an $\ell$-cyclic cover of smooth projective $\Ff_q$-curves and let
\[
C' = C\times_{\Ff_q} \Ff_{q^{n_q}}
\] be the scalar extension of $C$ to $\Ff_{q^{n_q}}$.
We start by a simple general observation connecting the number of $\Ff_q$-rational points of $C$ and $\Ff_{q^{n_q}}$-rational points of $C'$ lying above $\Pp^1(\Ff_q)$.

It is convenient to introduce the following piece of notation: For $x\in \Pp^1(\Ff_q)$, let
\begin{equation}\label{eq:Nx}
N_x = \# \{ y\in C(\Ff_q) : \pi(y) = x\} \quad \mbox{and} \quad N_x' = \# \{ y\in C'(\Ff_{q^{n_q}}) : \pi(y) = x\}
\end{equation}
be the number of $\Ff_q$-rational and $\Ff_{q^{n_q}}$-rational points on $C$ and $C'$ lying above $x$, respectively.

\begin{prop}\label{numptsprop}
Let $x\in \Pp^1(\Ff_q)$. Then $N_x=N_x'$.
\end{prop}

\begin{proof}
Let $x\in \mathbb{P}^1(\Ff_q)$, let $O=O_{\Pp^1_{\Ff_q},x}$ be the local ring at $x$ (which is the localization of $\Ff_q[X]$ at $X-x$ if $x$ is finite or of $\Ff_q[X^{-1}]$ at $X^{-1}$ if $x=\infty$) and $\frak p$ the corresponding maximal ideal (which is $(X-x)$ if $x$ is finite or $(X^{-1})$ if $x=\infty$).  Since $L=\Ff_q(C)$ is Galois with cyclic group of prime order $\ell$,
there are $3$ possible factorizations of $\frak p = \frak{P}_1^{e}\cdots \frak{P}_g^e$; namely,

\begin{enumerate}
\item $g=\ell$, $e=f=1$, (where  $f=\deg \frak P_i$) in which case $N_x=\ell$
\item $e=\ell$, $f=g=1$, in which case $N_x=1$,
\item $f=\ell$, $e=g=1$, in which case $N_x=0$.
\end{enumerate}
Now we base change to $\Ff_{q^{n_q}}$, and denote it as before by adding a tag. Since $n_q= [\Ff_{q^{n_q}}:\Ff_q]$ is co-prime to $\ell$ (as we add $\ell$-th root of unity), we get that $f=f'$. Since the extension $\Ff_{q^{n_q}}/\Ff_q$ is unramified, we get that $e=e'$. Since $\ell=[L:K]=[L':K']$, we conclude that $g=g'$. As $N_x'$ is determined by $e',f',g'$ in the same manner as $N_x$ is determined by $e,f,g$, we conclude that $N_x= N'_x$.
\end{proof}

Next we compute the number of rational points in terms of the parametrization of Section~\ref{genformsec}. Recall that $b\in \Ff_{q^{n_q}}^*$,
\begin{equation}\label{eq:DefF}
F=f_1f_2^2\cdots f_{\ell-1}^{\ell-1}
\end{equation} for $(f_1,\ldots, f_{\ell-1})\in \F_{n_q}(D)$ and that $C=C_{F,b}$ in the notation of \eqref{rmk:modelforcovers}. Also recall that $C'$ is birationally equivalent to the affine plane curve $Y^\ell = F_{\mathbf{v}_0}$, see \eqref{rmk:modelforcovers}, where $F_{\mathbf{v}_0}$ is given in \eqref{eq:DefFv}.

Let $\chi_\ell$ be a primitive multiplicative character of $\Ff_{q^{n_q}}$ of order $\ell$ (there exists such, as $n_q$ by definition is the minimal positive integer with $q^{n_q}\equiv 1\mod \ell$). For $x\in \Ff_q$ and for a polynomial $G$, we define
\begin{equation}\label{eq:chi_x}
\chi_x(G) = \chi_\ell(G(x)),
\end{equation}
which is a multiplicative Dirichlet character modulo $X-x$ of order $\ell$.
We extend this definition to $x=\infty\in \mathbb{P}^1(\Ff_q)$ by setting
\[
\chi_{\infty}(G) = \chi_{\ell}(g_n),
\]
with $g_n$ being the leading coefficient of $G$.

\begin{prop}\label{numptscor}
We have
\begin{equation}\label{analytic_formula}
\#C(\Ff_q) = \sum_{w=0}^{\ell-1}\chi_\ell^w(b) + \sum_{x\in\Ff_q} \sum_{w=0}^{\ell-1} \chi_{x}^w\left(F_{\mathbf{v}_0}\right).
\end{equation}
\end{prop}

\begin{proof}
 In  \cite{Meisner17} it is  established that for $x\in \mathbb{P}^1(\Ff_q)$
\[
N_x' = \sum_{w=0}^{\ell-1} \chi_x^{w}(F_{\mathbf{v}_0})
\]
(To see it follows from \emph{loc.\ cit.}, plug $\ell$, $F_{\mathbf{v}_0}$ for $r$, $F$ in the notation of \emph{loc.\ cit.} For finite $x$, we note that $F_{\mathbf{v}_0}(x) \neq 0$ since $f_i(x)\neq 0$ as they are $n_q$-divisible polynomial.
Thus the formula is given in the first paragraph of the proof of Lemma~2.1 in \emph{loc.\ cit.} For $x=\infty$, see page 536.)

Then, as $N_x=N_x'$ (Proposition~\ref{numptsprop}) we get that
\[
\#C(\Ff_q) = \sum_{x\in \mathbb{P}^1(\mathbb{F}_q)} N_x = \sum_{x\in \mathbb{P}^1(\mathbb{F}_q)} N_x' = \sum_{x\in\mathbb{P}^1(\Ff_q)} \sum_{w=0}^{\ell-1} \chi_{x}^w\left(F_{\mathbf{v}_0}\right).
\]
The leading coefficient of $F_{\mathbf{v_0}}$ is $b^{n_q}$ and so $\sum_{w=0}^{\ell-1}\chi_{\ell}^w(b) = \sum_{w=0}^{\ell-1}\chi_{\ell}^w(b^{n_q})$.
\end{proof}

\section{Set Count}\label{SetCount}
In light of the analytic formula \eqref{analytic_formula}, the study of the distribution of the number of points on $\ell$-cyclic covers parameterized by $\F_{n_q}(D)\times \Ff_{q^{n_q}}^*$
may be reduced to the computation of the size of  sets of the form
\begin{equation}\label{eq:FknqD}
\F^k_{n_q}(D)=\F^k_{n_q,\epsilon}(D) = \{(f_1,\dots,f_{\ell-1})\in\F_{n_q}(D) : \chi_{\ell}(F_{\mathbf{v}_0}(x_i)) = \epsilon_i, i=1,\dots,k\},
\end{equation}
where $x_1,\ldots, x_k \in\Ff_q$ are pairwise distinct elements, $\epsilon=(\epsilon_1,\ldots, \epsilon_{k})\in\mu_{\ell}^k$ is fixed, and, as usual, $F$ and $F_{\mathbf{v}_0}$, $\mathbf{v}_0 = (1,q^{-1}, \ldots, q^{1-n_q})$ are as defined in \eqref{eq:DefF} and \eqref{eq:DefFv}, respectively.
The computation is done by analyzing an appropriate  generating function $\G_k(u)$.

\begin{rem}

While the definition of $\F^k_{n_q}(D)$ depends on the choice of $b$ and of $(\epsilon_1,\dots,\epsilon_k)$, its asymptotic size --- which is computed below in \eqref{Fsize1} --- is independent of this choice. Hence we omit the $\epsilon_i$'s from the notation.

\end{rem}

\subsection{Generating Series}
Let $\zeta_{n_q} \in \mathbb{C}$ be a primitive $n_q$-th root of unity and define the following  auxiliary functions:
\[
\begin{split}
\mathcal{I}_{\infty}(F) &:=  \mu^2(F)\prod_{P|F}\frac{1}{n_q}\left(\sum_{i=0}^{n_q-1} \zeta_{n_q}^{i\deg(P)} \right)
\\
\mathcal{I}_{x_i}(F) &:= \frac{1}{\ell}\left(\sum_{w=0}^{\ell-1} \epsilon_{i}^{-w}\chi_{x_i}^w(F)\right),
\end{split}
\]
with the convention that $\chi_{x}^0$ is the trivial character modulo $X-x$.
Using these functions we define a function in $\ell-1$ variables (which are always assumed to be monic polynomials)
\[
\I(f_1,\ldots, f_{\ell-1}) = \I_{\infty}(f_1\cdots f_{\ell-1}) \prod_{i=1}^k\I_{x_i}(F_{\mathbf{v}_0}) ,
\]
with $F$ and $F_{\mathbf{v}_0}$ defined as in \eqref{eq:DefF} and \eqref{eq:DefFv}.

\begin{lem}
For a tuple of monic polynomials $(f_1,\ldots,f_{\ell - 1})$ we have
\[
\I(f_1,\ldots, f_{\ell-1}) = \begin{cases}
1, &  \mbox{if }(f_1,\ldots, f_{\ell-1})\in \F_{n_q}^k(\deg(f_1\cdots f_{\ell-1}))\\
0, & \mbox{otherwise.}
\end{cases}
\]
In particular,
\begin{align}\label{generatfun}
\G_k(u) := \sum_{f_1,\dots,f_{\ell-1}}\I(f_1,\ldots, f_{\ell-1})u^{\deg(f_1\cdots f_{\ell-1})}  = \sum_{D=0}^{\infty} |\F^k_{n_q}(D)|u^D,
\end{align}
where the first sum is over monic polynomials.
\end{lem}

\begin{proof}
By the orthogonality relations, for a tuple of monic polynomials $(f_1,\ldots, f_{\ell-1})$, we have that
$\mathcal{I}_{\infty}(f_1\cdots f_{\ell-1}) = 1$ if  $(f_1,\dots,f_{\ell-1})\in \F_{n_q}(D)$, and $=0$ otherwise. Similarly, if $(f_1,\ldots, f_{\ell-1})\in \F_{n_q}(D)$, then   $\I_{x_i} (F_{\mathbf{v}_0}) =1$ if $\chi_{x_i}(F_{\mathbf{v}_0})=\epsilon_i$ and $=0$ otherwise. This completes the proof of the lemma.
\end{proof}

Expanding the product $\prod_{i} \I_{x_i}(F)$ gives
\[
\prod_{i=1}^k \mathcal{I}_{x_i}(F) = \frac{1}{\ell^k} \prod_{i=1}^k \sum_{w=0}^{\ell-1} \epsilon_{i}^{-w}\chi_{x_i}^w(F)
 = \frac{1}{\ell^k}\sum_{\mathbf{w}}  \left(\prod_{i=1}^k \epsilon_i^{-w_i}\right) \chi_{\mathbf{w}}(F),
\]
where the sum is over all vectors $\mathbf{w}=(w_1,\dots,w_k)$ such that $0\leq w_i \leq \ell-1$
and
\begin{equation}\label{eq:chiw}
\chi_{\mathbf{w}}(F) = \prod_{i=1}^k \chi_{x_i}^{w_i}(F).
\end{equation}
This is a Dirichlet character of modulo $\prod (X-x_i)$ and is non-trivial if $\mathbf{w} \neq \vec{0}$.
Thus we get a decomposition of $\G_k(u)$,
\begin{align}\label{eq:factorization_generating_function}
\G_k(u) = \frac{1}{\ell^k}\sum_{\mathbf{w}} \left(\prod_{i=1}^k \epsilon_i^{-w_i}\right) \G_{\mathbf{w}}(u),
\end{align}
where
\begin{equation}\label{Gw}
\G_{\mathbf{w}}(u) = \sum_{f_1,\dots, f_{\ell-1}} \I_{\infty}(f_1\cdots f_{\ell-1}) \chi_{\mathbf{w}}(F_{\mathbf{v}_0}) u^{\deg(f_1...f_{\ell-1})}.
\end{equation}

\subsection{Euler Product}
Recall that a nonzero multivariate function $\psi$ in  is called \emph{firmly multiplicative} if
\[
\psi(a_1b_1,\ldots, a_rb_r)=\psi(a_1,\ldots, a_r)\psi(b_1,\ldots, b_r)
\]
whenever $\gcd(a_i,b_i)=1$, for all $i=1,\ldots, r$. These functions are determined by their values on $(1,\ldots, 1, p^\alpha,1\ldots, 1)$ and behave as multiplicative functions in one variable. We assume that the reader is familiar with the standard properties of firmly multiplicative functions, or at least that the reader may complete the details of how those are derived from the one variable case; if this is not the case, one may consult the survey paper \cite{Toth}.

\begin{prop}
As usual, let $K=\Ff_q(X) = \Ff_q(\Pp^1_{\Ff_q})$ and $K'=\Ff_{q^{n_q}}(X)$. For each prime $P$ of $K$ we fix a prime $\mathfrak{P}$ of $K'$ lying above $P$. Then we have the Euler decomposition
\begin{align}\label{eulprod}
\G_{\mathbf{w}}(u) = \prod_{\substack{P \\ n_q|\deg(P)}} \left(1+\sum_{j=1}^{\ell-1}\chi^{j}_{\mathbf{w}}(\mathfrak{P})u^{n_q \deg(\frak P)}\right),
\end{align}
with $\G_{\mathbf{w}}$ and $\chi_{\mathbf{w}}$ as defined in \eqref{Gw} and \eqref{eq:chiw}.
\end{prop}

\begin{rem}\label{rem:sumchi}
If $\frak P$ and $\frak P'$ are two primes over $P$, then $\phi^k(\frak P) = \frak P'$ for some $k$, and so by \eqref{eq:phiqx}, the inner sum in \eqref{eulprod} equals to the same sum with $\frak P'$ replacing $\frak P$.
\end{rem}

\begin{proof}
The function
\[
\Psi (f_1,\ldots, f_{\ell-1}) = \I_{\infty}(f_1\cdots f_{\ell-1}) \chi_{\mathbf{w}}(F_{\mathbf{v}_0})
\]
is firmly multiplicative, hence to compute the Euler  decomposition of
$\G_{\mathbf{w}}(u)$, it suffices to evaluate $\Psi$ on
\[
e_{i}(P^{\nu})=(1,\ldots, 1,P^\nu,1\ldots, 1),
\]
with $P^\nu$ appearing in the $i$-th place and $P$ is a prime polynomial in $K$.

Since $\I_{\infty}(P^\nu)=1$ only when $\nu=1$ and $n_q\mid \deg(P)$ (and $=0$ otherwise), we may restrict to this case.

By Corollary~\ref{basefieldcor}, $P$ has a $\mu_{\ell}$-stable factorization
\[
P = \mathfrak{P}_1 \cdots \mathfrak{P}_{n_q},
\]
in which we may assume without loss of generality that $\mathfrak{P}_1=\mathfrak{P}$.
We note that as $P$ is prime in $K$, the $\mathfrak{P}_i$ are prime in $K'$.
Then $F = P^{i}$, $F_{j} = \frak P_j^i$, and
\[
F_{\mathbf{v}_0}=\prod_{j=1}^{n_q} \frak P_j^{iv_j},
\]
with $v_{j}\equiv q^{1-j} \mod \ell$. By \eqref{eq:phiqx},  $\frak P_j^{iv_j}(x_r)$ equals to  $\frak  P(x_i)^{i}$ up to $\ell$-th-powers. As $\chi_{\mathbf{w}}$ is defined modulo $\prod (X-x_i)$ and is trivial on $\ell$-th-powers, we conclude that
\[
\Psi(e_i(P)) = \chi_{\mathbf{w}} (F_{\mathbf{v}_0}) = \chi_{\mathbf{w}}^{in_q}(\frak P).
\]
Thus we get the following Euler  decomposition:
\[
\G_{\mathbf{w}}(u) =
 \prod_{\substack{P \\ n_q|\deg(P)}} \left(1+\sum_{i=1}^{\ell-1}\chi^{in_q}_{\mathbf{w}}(\mathfrak{P})u^{\deg(P)}\right).
\]
By Lemma~\ref{basefieldlem}, $\deg (P) = n_q \deg (\mathfrak P)$. Thus, since $(n_q,\ell)=1$, we have that $\sum_{i=1}^{\ell-1}\chi^{in_q}_{\mathbf{w}}=\sum_{j=1}^{\ell-1}\chi^{j}_{\mathbf{w}}$, and the proof is done.
\end{proof}

\subsection{Analytic Continuation}\label{analcont}

For a non-trivial character $\chi$ of $\Ff_{q^{n_q}}[X] \subseteq K'$ we define
\begin{equation}\label{eq:Lfn}
L_{K'}(u,\chi) = \prod_{\mathfrak{P}} \left( 1-\chi(\mathfrak{P})u^{\deg(\mathfrak{P})} \right)^{-1},
\end{equation}
where the product is over all prime polynomials in $K'$. Then, $L_{K'}(u,\chi)$ is a polynomial (due to the orthogonality relations) and its zeros lie on the circle $|u|=q^{-n_q/2}$ (due to the Riemann Hypothesis for curves).
We shall use repeatedly that a product
\begin{equation}\label{eq:radiusconvegence}
\prod_{P}(1+O(u^{c\deg P}))
\end{equation}
over prime polynomials absolutely converges in the disc $|u|<q^{-1/c}$.

\begin{prop}\label{analcontprop}
For any $\mathbf{w}$, there exists a non-vanishing function $H_{\mathbf{w}}(u)$ which is analytic in the open disc $|u|<q^{-1/2}$ such that
\begin{enumerate}
\item if $\mathbf{w}\neq \vec{0}$, then
\[
\G_{\mathbf{w}}(u) = \prod_{j=1}^{\ell-1}\left(\frac{L_{K'}(u^{n_q},\chi^j_{\mathbf{w}})}{L_{K'}(u^{2n_q},\chi^{2j}_{\mathbf{w}})}\right)^{\frac{1}{n_q}} H_{\mathbf{w}}(u).
\]
\item if $\mathbf{w} = \vec{0}$, then
\[
\G_{\vec{0}}(u) = \prod_{j=0}^{n_q-1}\left(1-\zeta_{n_q}^jqu \right)^{-\frac{\ell-1}{n_q}} H_{\vec{0}}(u),
\]
and $H_{\vec{0}}(\zeta_{n_q} u)=   H_{\vec{0}}(u)$.
\end{enumerate}
In particular,
\begin{enumerate}
\item[(3)]
$\G_{\mathbf{w}}(u)$ has a meromorphic continuation to $|u|<q^{-1/2}$ which is analytic if $\mathbf{w}\neq \vec0$ and has poles of order $\frac{\ell-1}{n_q}$ at $u=(\zeta_{n_q}^jq)^{-1}$, $j=0,\ldots, n_q-1$ if $\mathbf{w}=\vec{0}$.
\end{enumerate}
\end{prop}

\begin{proof}
By the Riemann Hypothesis, the poles and zeros of $\frac{L_{K'}(u^{n_q},\chi^j_{\mathbf{w}})}{L_{K'}(u^{2n_q},\chi^{2j}_{\mathbf{w}})}$ lies on the circles $|u|=q^{-1/4}$ and $|u|=q^{-1/2}$, respectively, hence it has an analytic $n_q$-th root in the open disc $|u|< q^{-1/2}$ and so (3) follows from (1) and (2).

Now we prove (1).
Since each irreducible $P$ in $\Ff_q[X]$ with $n_q \mid \deg (P)$ has $n_q$ prime divisors in $\Ff_{q^{n_q}}[X]$ (Lemma~\ref{basefieldlem}) we get by \eqref{eulprod} applied to all the primes dividing $P$ that
\[
\G_{\mathbf{w}}(u)^{n_q}  = \prod_{\substack{P \\ n_q\mid \deg(P)}} \prod_{\frak P\mid P} \left(1+\sum_{j=1}^{\ell-1}\chi^{j}_{\mathbf{w}}(\mathfrak{P})u^{n_q\deg(\mathfrak{P})}\right).
\]
We take $n_q$-th root and break this product into two parts: Denote
\[
F_1(u)  = \prod_{{\substack{P \\ n_q\nmid\deg(P)}}}\prod_{\mathfrak{P}|P} \left(1+\sum_{j=1}^{\ell-1}\chi^{j}_{\mathbf{w}}(\mathfrak{P})u^{n_q\deg(\mathfrak{P})}\right)^{-\frac{1}{n_q}}
\]
and
\[
F_2(u) = \prod_{\mathfrak{P}} \left(1+\sum_{j=1}^{\ell-1}\chi^{j}_{\mathbf{w}}(\mathfrak{P})u^{n_q\deg(\mathfrak{P})}\right)^{\frac{1}{n_q}}
\]
so
\begin{align}\label{factorization}
\G_{\mathbf{w}}(u)  = F_1(u)F_2(u).
\end{align}
We study each of the factors starting with $F_1$. By Lemma~\ref{basefieldlem} and Remark~\ref{rem:sumchi} for $|u|<q^{-1/2}$,
\begin{align*}
F_1(u) &=  \prod_{\substack{P \\ n_q\nmid\deg(P)}} \left(1+\sum_{j=1}^{\ell-1}\chi^{j}_{\mathbf{w}}(\mathfrak{P})u^{\frac{n_q}{(\deg(P),n_q)}\deg(P)}\right)^{-\frac{(\deg(P),n_q)}{n_q}}.
\end{align*}
Since for $n_q\nmid \deg P$ we have
\[
\left(1 + \chi^{j}_{\mathbf{w}}(\mathfrak{P})u^{\frac{n_q}{(\deg(P),n_q)}\deg(P)}\right)^{-\frac{(\deg(P),n_q}{n_q}}=1+O(q^{2\deg P}),
\]
by \eqref{eq:radiusconvegence}, we get that $F_1(u)$ is analytic in this open disc $|u|<q^{-1/2}$.

Next we study $F_2$. Expanding the product in its definition we get
\begin{equation}\label{eqF2}
F_2(u) = \prod_{j=1}^{\ell-1} \prod_{\mathfrak{P}} \left(1 + \chi^{j}_{\mathbf{w}}(\mathfrak{P})u^{n_q\deg(\mathfrak{P})}\right)^{\frac{1}{n_q}}  \prod_{\mathfrak{P}} \left(1 + O(u^{2 n_q \deg (\mathfrak P)}) \right)
\end{equation}
By \eqref{eq:radiusconvegence}, the right product absolutely  converges in $|u|<q^{-1/2}$ and  by \eqref{eq:Lfn},
\[
\prod_{\mathfrak{P}} \left(1 + \chi^{j}_{\mathbf{w}}(\mathfrak{P})u^{n_q\deg(\mathfrak{P})}\right) = \frac{L_{K'}(u^{n_q},\chi^j_{\mathbf{w}})}{L_{K'}(u^{2n_q},\chi^{2j}_{\mathbf{w}})},
\]
which together with the analyticity of $F_1$, with \eqref{eqF2}, and with \eqref{factorization} finishes the proof of (1).

For (2), we use \eqref{eulprod}, to get that
\[
\G_{\vec{0}}(u)  = \prod_{\substack{P \\ n_q|\deg(P)}} \left(1+(\ell-1)u^{\deg(P)}\right).
\]
By the orthogonality  relations $1_{\{n\equiv 0\mod n_q\}} = \frac{1}{n_q} \sum_{j=0}^{n_q-1} \zeta_{n_q}^{jn}$ we have
\[
\G_{\vec{0}}(u) = \prod_{P} \Big(1 + \frac{\ell-1}{n_q} \sum_{j=0}^{n_q-1}\zeta_{n_q}^{j\deg(P)} u^{\deg(P)} \Big).
\]
Then we define $H_{\vec{0}} (u)$ by the equation
\begin{align*}
\G_{\vec0}(u) &= \prod_{j=0}^{n_q-1}\left( \prod_{P} \left(1 - (\zeta_{n_q}^ju)^{\deg(P)}\right)\right)^{-\frac{\ell-1}{n_q}} H_{\vec{0}}(u)\\
& = \prod_{j=0}^{n_q-1}\left(\frac{1}{1-q\zeta_{n_q}^ju} \right)^{\frac{\ell-1}{n_q}} H_{\vec{0}}(u).
\end{align*}
As $H_{\vec{0}}(u) = \prod_{P} (1+O(u^{2\deg(P)}))$ holds true,  \eqref{eq:radiusconvegence} implies that $H_{\vec{0}}(u)$ is analytic in $|u|<q^{-1/2}$. Further, $H_{\vec{0}}(\zeta_{n_q}u) = H_{\vec{0}}(u)$ since both $\G_{\mathbf{w}}$ and $\prod_{j=0}^{n_q-1}\left(\frac{1}{1-q\zeta_{n_q}^ju} \right)^{\frac{\ell-1}{n_q}}$ satisfy this relation.
\end{proof}

\subsection{Computing $|\F^k_{n_q}(D)|$}

We determine the size of the set $\F^k_{n_q}(D)$ using the residue theorem applied to the contour $C_{\epsilon} = \{u : |u|=q^{-1/2-\epsilon}\}$ and to the function $\G_k(u)/u^{D+1}$. By Proposition~\ref{analcontprop}, by \eqref{generatfun}, and by \eqref{eq:factorization_generating_function} we have
\begin{equation}\label{eq:contourint}
\begin{split}
\frac{1}{2\pi i} \oint_{C_{\epsilon}}\frac{\G_k(u)}{u^{D+1}} du & = \Res_{u=0} \left(\frac{\G_k(u)}{u^{D+1}}\right) + \sum_{j=0}^{n_q-1} \Res_{u=(\zeta_{n_q}^jq)^{-1}} \left(\frac{\G_k(u)}{u^{D+1}}\right)\\
& = |\F^k_{n_q}(D)| +\frac{1}{\ell^k} \sum_{j=0}^{n_q-1}\Res_{u=(\zeta_{n_q}^jq)^{-1}} \left(\frac{\G_{\vec{0}}(u)}{u^{D+1}}\right).
\end{split}
\end{equation}

We compute each of the terms separately.
Put $N=\frac{\ell-1}{n_q}-1$ and $u_j = (\zeta_{n_q}^jq)^{-1}$. Since $\G_{\vec0}$ has a pole at $u=u_j$ of order $N+1$ and using Proposition~\ref{analcontprop}, we have

\begin{align}\label{zerosatnonzeropoints}
\begin{split}
N!\cdot\Res_{u=u_j} \left(\frac{\G_{\vec{0}}(u)}{u^{D+1}}\right)
	& = \frac{d^N}{du^N} \left[ (u-u_j)^{N+1}\frac{\G_{\vec{0}}(u)}{u^{D+1}} \right] \bigg|_{u=u_j}\\
	& = \frac{d^{N}}{du^{N}}\Bigg[ \frac{u_j^{N+1}}{u^{D+1}} \prod_{\substack{i=0 \\ i\not=j}}^{n_q-1}u_i^{N+1}\left(u_i-u\right)^{-(N+1)} H_{\vec{0}}(u)\Bigg]\Bigg|_{u=u_j}\\
	&= \sum_{m=0}^{N} \binom{N}{m} \frac{d^m}{du^m} \frac{u_j^{N+1}}{u^{D+1}}\Bigg|_{u=u_j}  \frac{d^{N-m}}{du^{N-m}} \Bigg[\prod_{\substack{i=0 \\ i\not=j}}^{n_q-1}\left(\frac{1}{1-u/u_i}\right)^{N+1} H_{\vec{0}}(u)\Bigg]\Bigg|_{u=u_j} \\
= & -\frac{1}{n_q}P(D) u_j^{-D}= -\frac{1}{n_q} P(D)q^D\zeta_{n_q}^{jD},
\end{split}
\end{align}
where $P(D)$ is a polynomial of degree $N$ which is independent of $j$.

Since $\G_k(u)$ is continuous on $C_{\epsilon}$, it is $O(1)$ there, and we have
\[
\left| \frac{1}{2\pi i} \oint_{C_{\epsilon}}\frac{\G_k(u)}{u^{D+1}} du \right| = O\left(q^{(1/2+\epsilon)D}\right).
\]
Plug this bound and  \eqref{zerosatnonzeropoints} in \eqref{eq:contourint} to get
\begin{align}\label{Fsize1}
|\F^k_{n_q}(D)| = \begin{cases} \frac{1}{\ell^k}P(D)q^D + O\left(q^{(1/2+\epsilon)D}\right), & D\equiv 0 \mod{n_q} \\ 0, & \mbox{otherwise}.\end{cases}
\end{align}

We emphasize that the main term is independent on the actual choice of the values $\epsilon_1,\ldots, \epsilon_k$ of the characters. The error term may depend on this choice.

\section{Proof of Theorem~\ref{thm:main}}\label{sec:proof}
Fix $D\equiv 0 \mod n_q$.
Let $N\geq 0$ and  $\epsilon_0,\epsilon_1,\ldots, \epsilon_{q}\in \mu_{\ell}$ such that
\[
\sum_{i=0}^{q} \sum_{w=0}^{\ell-1} \epsilon_i^w = N.
\]

Conditioning on $b$, \eqref{Fsize1} gives that the probability $\pi$ that $\chi_{x_i}(F_{\mathbf{v}_0}) = \epsilon_i$ for $i=1,\ldots, q$ and $\chi_{\ell}(b)=\epsilon_0$ is
\[
\pi = \frac{1}{\ell^{q+1}} + O(q^{D(-1/2+\epsilon)}).
\]
Thus, \eqref{analytic_formula} implies that
\[
\Prob(\#C_{F,b}(\Ff_q) = N) = \sum_{\substack{\epsilon_0,\ldots, \epsilon_q\in \mu_{\ell}\\ \sum_i\sum_w \epsilon_i^w = N}} \frac{1}{\ell^{q+1}} + O(q^{D(-1/2+\epsilon)}).
\]
Putting $\epsilon_i' = \sum_{w=0}^{\ell-1} \epsilon_i^w$, we have that $\epsilon'_i\in \{0,\ell\}$. Moreover,  $\epsilon'_i=0$ if and only if $\epsilon_i\neq 1$. This means that if $\epsilon'_i=0$ there are $\ell-1$ different choices for $\epsilon_i$, whereas if $\epsilon'_{i}=\ell$, then there is one choice: $\epsilon_i=1$. We plug this in the previous equation and get
\[
\begin{split}
\Prob(\#C_{F,b}(\Ff_q) = N)
	&= \sum_{\substack{\epsilon_0',\ldots, \epsilon_q'\in \{0,\ell\}\\ \sum_i\epsilon_i' = N}} \frac{(\ell-1)^{\#\{i:\epsilon'_i=0\}}}{\ell^q} + O(q^{D(-1/2+\epsilon)})\\
	&=\sum_{\substack{\epsilon_0',\ldots, \epsilon_q'\in \{0,\ell\}\\ \sum_i\epsilon_i' = N}} \left(\frac{\ell-1}{\ell}\right)^{\#\{i:\epsilon'_i=0\}}\left(\frac{1}{\ell}\right)^{\#\{i:\epsilon'_i=\ell\}}
	 + O(q^{D(-1/2+\epsilon)})\\
	 &= \Prob(\sum_{i=0}^q X_i = N),
\end{split}
\]
with $X_i$ as in the formulation of the theorem. Since  by Proposition~\ref{H_gtoF_n}
\[
\Prob(\#C_{F,b}(\Ff_q) = N) =\Prob(\#C(\Ff_q)=N)
\] 
with $C$ random curve in $\mathcal{H}_{g,\ell}$, this finishes the proof.
\qed

\bibliography{SecondDraft}
\bibliographystyle{amsplain}

\end{document}